\documentclass[reqno,10pt]{article}
\usepackage{graphicx}
\usepackage{amsmath}
\usepackage{amssymb}
\usepackage{xcolor}
\usepackage{amsthm}
\usepackage{amsfonts}
\usepackage{mathtools}
\usepackage{enumitem}
\usepackage{bbm}
\usepackage[l2tabu,orthodox]{nag}
\usepackage[all,error]{onlyamsmath}
\usepackage[strict]{csquotes}
\usepackage{url}
\newtheorem{theorem}{Theorem}
\newtheorem{lemma}[theorem]{Lemma}
\newtheorem{corollary}[theorem]{Corollary}
\theoremstyle{definition}
\newtheorem{definition}[theorem]{Definition}
\newtheorem{remark}[theorem]{Remark}
\newtheorem{example}[theorem]{Example}

\newtheorem{proposition}[theorem]{Proposition}

\numberwithin{equation}{section}
\numberwithin{theorem}{section}

\newenvironment{OMabstract}{\noindent\textbf{Abstract.} }{\medskip}
\newenvironment{OMsubjclass}{\noindent\textbf{Mathematics Subject Classification (2020):} }{\medskip}
\newenvironment{OMkeywords}{\noindent\textbf{Keywords:}  }{\medskip}

\begin{document}

\author{Katherine Rossella Foglia, Vittorio Colao} %Please write full names and surnames of all co-authors here.
\title{On the Rate of Asymptotic Regularity of Iterative Methods for Nonexpansive Mappings in CAT(0) Spaces and Hyperbolic Optimization}
\maketitle

%\hrule\begin{center}Submission for publication in \textsc{Opuscula Mathematica}\end{center}\hrule\bigskip

\begin{OMabstract}
    %------------------------- Please type your abstract here ------------------
The Krasnosel'ski\u{\i}–Mann and Halpern iterations are classical schemes for approximating fixed points of nonexpansive mappings in Banach spaces, and have been widely studied in more general frameworks such as $CAT(\kappa)$ and, more generally, geodesic spaces. 
Convergence results and convergence rate estimates in these nonlinear settings are already well established. 
The contribution of this paper is twofold: first, we extend to complete $CAT(0)$ spaces proof techniques originally developed in the linear setting of Banach and Hilbert spaces, thereby recovering the same asymptotic regularity bounds; second, we introduce a Halpern-type optimizer for hyperbolic optimization as a nonlinear counterpart of the Euclidean HalpernSGD scheme.

%The contribution of this paper is to extend to complete $CAT(0)$ spaces the proof techniques originally developed in the linear setting of Banach and Hilbert spaces, thereby recovering the same asymptotic regularity bounds and to introduce a novel optimizer for Hyperbolic Deep learning based on Halpern Iteration similarly to HalpernSGD \cite{foglia2024halpernsgd,colao2025optimizer} in Euclidean setting.
    %---------------------------------------------------------------------------
\end{OMabstract}

\begin{OMkeywords}
    %------------------------- Please type your keywords here ------------------
    Geodesic Spaces, Hadamard Spaces, Metric Fixed Point Theory, Optimization, Hyperbolic Deep Learning, Halpern Iterates, Krasnosel'ski\u{\i}–Mann iterates.
    %---------------------------------------------------------------------------
\end{OMkeywords}

\begin{OMsubjclass}
    %------------------------- Please type your 2020 MSC numbers here ----------
    47H10, 46N10, 53C25, 58C30.
    %---------------------------------------------------------------------------
\end{OMsubjclass}

%------------------- Type contents of your article below --------------------

\section{Introduction}
Deep representation learning in hyperbolic spaces has recently gained significant attention due to the ability of hyperbolic geometry to accommodate tree‑like structures and more generally, to effectively model hierarchically structured data \cite{peng2021hyperbolic}.
In \cite{ganea2018hyperbolicNN}, the authors showed that core neural network operations, such as multinomial logistic regression, feed-forward networks, and recurrent neural networks, can be reformulated in hyperbolic spaces. 
These developments enable the direct embedding of sequential data and the training of classifiers in hyperbolic spaces. 
As remarked in \cite{ganea2018hyperbolicNN}, until recently it was common practice to embed data in the Euclidean space $\mathbb{R}^n$ due to the availability of closed-form formulas for distance and inner product, but this is not sufficient. 
In fact, many datasets reside on manifolds or graphs with inherently non-Euclidean geometry \cite{bronstein2017geometric}.
For instance, no matter how high the dimension, Euclidean space cannot embed arbitrary tree structures with negligibly low distortion, while a two-dimensional hyperbolic space easily can \cite{sala2018representation}.
Therefore, the so-called {hyperbolic neural architectures can potentially lead to drastically more compact models than their Euclidean counterparts \cite{peng2021hyperbolic}.
Training these models, however, requires optimizers adapted to the geometric setting, and one of the most widely used is Riemannian stochastic gradient descent (RSGD) \cite{bonnabel2013stochastic}. 

As observed in \cite{colao2025optimizer, foglia2024halpernsgd}, Gradient Descent (GD) can be seen as a special case of the Krasnosel'ski\u{\i}--Mann (KM) iteration. 
The same works also point out the existence of alternative schemes, such as the Halpern and Viscosity iterations, which achieve faster convergence rates than KM and GD. 
These methods have been studied in both linear and nonlinear settings, and can be used to design faster GD-type optimizers for hyperbolic deep neural networks, as already done in the linear case for HalpernSGD \cite{foglia2024halpernsgd}.

In particular, the rates of asymptotic regularity for both the Krasnosel’skii--Mann and Halpern iteration have been established also in the nonlinear setting. Our contribution in this work is to extend to specific geodesic spaces the proof techniques originally developed in the linear setting, thereby recovering the same rates of convergence. 
It is worth remarking that part of the results we present in this paper have already been investigated in the literature, although with different approaches and proofs.
The aim of this paper is therefore to develop a theoretical understanding of nonlinear optimization and to set the stage for a Hyperbolic version of HalpernSGD.

The paper is organized as follows. Section~\ref{sec:preliminari} provides the mathematical background needed for our analysis. 
Section~\ref{sec:statoarte} reviews existing convergence results for KM and Halpern-type iterations in different geometric settings. 
In Section~\ref{sec:ContributionIperbolici}, we present and prove our main results on convergence rates in specific geodesic spaces, including hyperbolic ones. 
Section~\ref{sec:iperbolicHalpernGD} provides the theoretical construction of a hyperbolic HalpernGD-type optimizer.

Finally, Section~\ref{sec:conclusione} outlines future work and contains our concluding remarks.

\section{Preliminaries}
\label{sec:preliminari}
%We recall some fundamental notions regarding geodesic metric spaces following \cite{bridson2013metric}.
In this section, we introduce the necessary background on geodesic spaces following \cite{bridson2013metric}.

\medskip

\begin{definition} \label{eq:geodesicProperty}
    A metric space \((X,d)\) is called a geodesic space if for every pair of distinct points \(x,y \in X\), there exist a real number \(l>0\) and a continuous map \(\gamma : [0,l] \to X\) such that:
    \[
    \gamma(0)=x,\qquad \gamma(l)=y,
    \]
    and
    \[
    d(\gamma(t_1),\gamma(t_2)) = |t_1-t_2|
    \quad \text{for all } t_1,t_2\in[0,l].
    \]
\end{definition}

Such a map \(\gamma\) is called a geodesic path, or more briefly a geodesic, joining $x$ and $y$ and, in other words, it is an isometry between the real interval $[0,l]$ and its image $\gamma([0,l])$, which in turn is called a geodesic segment with endpoints $x$ and $y$.
In particular, taking \(t_1=0\) and \(t_2=l\), one obtains \(l=d(x,y)\). Therefore, by a linear reparametrization, geodesics may equivalently be defined on the interval \([0,1]\).
If $\gamma$ is defined on $[0,+\infty)$ it is called a geodesic ray; if it is defined on the whole $\mathbb{R}$ it is called a geodesic line; finally, if it is defined on an interval $I\subset\mathbb{R}$, not necessarily closed, and for every $t\in I$ there exists $\varepsilon>0$ such that the restriction of $\gamma$ to $I \cap (t-\frac{\varepsilon}{2},t+\frac{\varepsilon}{2})$ is a geodesic it is called a local geodesic.

In particular, for a real constant $r>0$, $(X, d)$ is said to be an $r$-geodesic space if every pair of points $x,y\in X$ with $d\big(x,y\big)<r$ can be joined by a geodesic. If $r= \infty$, an $\infty$-geodesic space corresponds to a geodesic space as in Definition \ref{eq:geodesicProperty}.
Important examples of geodesic spaces include:

\begin{example}
For \(n\in \mathbb{N}\) and \(\kappa\in \mathbb{R}\), the model space \(M_\kappa^n\) is defined by
\[
M_\kappa^n=
\begin{cases}
\mathbb{H}^n, & \text{if } \kappa<0,\\[0.3em]
\mathbb{E}^n, & \text{if } \kappa=0,\\[0.3em]
\mathbb{S}^n, & \text{if } \kappa>0,
\end{cases}
\]
endowed with the metric
\[
d_{M_\kappa^n}=
\begin{cases}
\dfrac{1}{\sqrt{-\kappa}}\, d_H, & \text{if } \kappa<0,\\[0.8em]
d_E, & \text{if } \kappa=0,\\[0.3em]
\dfrac{1}{\sqrt{\kappa}}\, d_S, & \text{if } \kappa>0,
\end{cases}
\]
where:
\begin{itemize}
    \item \(\mathbb{E}^n\) is the Euclidean space \(\mathbb{R}^n\) with the usual Euclidean distance \(d_E\);
    
    \item \(\mathbb{S}^n\subset \mathbb{R}^{n+1}\) is the unit sphere with the angular metric
    \[
    d_S(x,y)=\arccos(\langle x,y\rangle);
    \]
    
    \item \(\mathbb{H}^n\) is the hyperbolic space
    \(
    \mathbb{H}^n:=\left\{u\in\mathbb{R}^{n+1}\,\middle|\,\langle u,u\rangle_L=-1,\ u_0>0\right\},
    \)
    where
    \[
    \langle u,v\rangle_L:=-u_0v_0+\sum_{i=1}^n u_iv_i,
    \]
    and the hyperbolic metric \(d_H\) is given by
    \(
    \cosh(d_H(x,y))=-\langle x,y\rangle_L.
    \)
\end{itemize}

With this definition, \((M_\kappa^n,d_{M_\kappa^n})\) is a geodesic space (see \cite[p.23, Definition 2.10]{bridson2013metric}) and, according to the sign of \(\kappa\), the geodesics are Euclidean segments, arcs of great circles, or hyperbolic geodesics.

\end{example}

We say that \((X,d)\) is a uniquely geodesic space if for every pair of points \(x,y \in X\), there exists exactly one geodesic joining them. In such spaces, the geodesic segment with endpoints \(x\) and \(y\) is denoted by \([x,y]\). An important family of examples of uniquely geodesic spaces include (see \cite{bridson2013metric}):
%An important family of examples of uniquely geodesic spaces is given by the model spaces \(M_\kappa^n\) (see \cite{bridson2013metric}):

%\begin{example}[\textbf{\emph{The octant of a sphere}}]
    % We remark that $\mathbb{S}^n$, i.e. $M_\kappa^n$ for $\kappa>0$ is not uniquely geodesic, in fact, the geodesics between two antipodal points of the sphere are all the semi-circumferences with those two points as endpoints. However, the geodesic becomes unique if the distance of the points considered is less than $\frac{\pi}{\sqrt{\kappa}}$.
%\end{example}
%\begin{example} [\textbf{\emph{The Hyperbolic and Euclidean spaces}}]
    %As stated in \cite[Proposition 2.11]{bridson2013metric}, $M_\kappa^n$ is uniquely geodesic for $\kappa\leq0$.
%\end{example}

\begin{example}
    The model spaces $M_\kappa^n$ provide basic examples in the theory of geodesic spaces. As stated in \cite[Proposition 2.11]{bridson2013metric}, $M_\kappa^n$ is uniquely geodesic for $\kappa\leq 0$. On the other hand, when $\kappa>0$, that is, in the spherical case, $M_\kappa^n$ is not uniquely geodesic in general. Indeed, antipodal points on the sphere can be joined by infinitely many geodesics, namely all the semicircumferences having those points as endpoints. However, if the distance between two points is less than $\frac{\pi}{\sqrt{\kappa}}$, then the geodesic joining them is unique.
\end{example}

\medskip

Let \((X,d)\) be a geodesic space and \(x,y \in X\), we say that a point \(z_t \in X\) belongs to a geodesic segment with endpoints $x$ and $y$ if there exist \(t \in [0,1]\) such that:
\begin{equation}
    \label{AppartenenzaAllaGeodetica}
    d(z_t,x)=t\,d(x,y)\quad\text{and}\quad d(z_t,y)=(1-t)\,d(x,y).
\end{equation}
Fixing a specific geodesic segment with endpoints $x$ and $y$, we use the notation
\begin{equation}
\label{oplus}
    z_t = (1-t)x \oplus ty.
\end{equation}
In particular, for $t = \tfrac{1}{2}$ the point $z_t$ is called the midpoint of $x$ and $y$, and is denoted by $m(x,y)$.

\begin{remark}   \label{remarkPlus}
If $(X,d)$ is uniquely geodesic, then there exists a unique geodesic segment with endpoints $x$ and $y$. In this case, there is no ambiguity in writing $z_t \in [x,y]$, and it is not necessary to specify a particular geodesic segment in order to guarantee the uniqueness of $z_t$.
\end{remark}

%An immediate remark is that if \((X,d)\) is a uniquely geodesic space then for each \(t \in [0,l]\) there exist a unique point \(z\) belonging to the segment \([x,y]\).

A further important class of geodesic spaces is given by CAT(\(\kappa\)) spaces, introduced to generalize curvature conditions to metric spaces. In order to formally introduce them we provide some important definitions (see \cite[p.158]{bridson2013metric}):
\begin{definition} 
    In a geodesic space $(X,d)$ a geodesic triangle $\Delta([p,q],[q,r],[r,p])$ is characterized by three distinct points $p,q,r \in X$, called the vertices, and by the choices of three geodesic joining them, whose geodesic segments we denote by $[p,q],[q,r]$, $[r,p]$ and called the sides of the triangle.
\end{definition}
Even if the space is not uniquely geodesic, we can fix a single geodesic for each pair of vertices, so that the notation $[x,y]$ remains unambiguous. In the uniquely geodesic case, this choice is unnecessary and the triangle may be written simply as $\Delta(p,q,r)$. We say that a point $x$ belongs to $\Delta([p,q],[q,r],[r,p])$ if it lies on one of its sides and the perimeter is then the sum of the three side lengths, as the notion of distance remains well defined.
\begin{definition}%[\textbf{\emph{Comparison triangle}}]
    Let $(X,d)$ be a uniquely geodesic space, for a triangle $\Delta(p,q,r)$ in $(X,d)$, a comparison triangle $\overline{\Delta}(\overline{p},\overline{q},\overline{r})$ in the model space $M_\kappa^2$ is a triangle such that: $d_{M_\kappa^2}(\overline{p},\overline{q})=d(p,q)$, $d_{M_\kappa^2}(\overline{q},\overline{r})=d(q,r)$ and $d_{M_\kappa^2}(\overline{r},\overline{p})=d(r,p)$.
\end{definition}

Similarly, for $x\in \Delta$, a comparison point $\overline{x}\in \overline{\Delta}$ is a point lying on the side of $\overline{\Delta}$ corresponding to the side of $\Delta$ to which $x$ belongs, and has the same distance from the corresponding vertices as $x$ has from its own. Let $D_\kappa$ be the following quantity,
\[
    D_\kappa :=
    \begin{cases}
     \frac{\pi}{\sqrt{\kappa}} & \text{if } \kappa > 0 \\
    \infty & \text{if } \kappa \leq 0
    \end{cases} \text{ ,} 
\]
in \cite[p.24, Lemma 2.14]{bridson2013metric}, it is proved that in a uniquely geodesic space $(X,d)$, for any geodesic triangle $\Delta(p,q,r) \subset X$ whose perimeter is strictly less than $2D_\kappa$, a comparison triangle exists and it is unique up to isometry of $M_\kappa^2$.

 Having properly introduced the notions of geodesic triangle and comparison triangle, we can now provide the definition of the $CAT(\kappa)$ space as in \cite[p.158, Definition 1.1]{bridson2013metric}: 
\begin{definition} % [\textbf{\emph{$CAT(\kappa)$ spaces}}]
\label{Cat(k) spaces}
   Let $(X,d)$ be a geodesic space, let $\kappa$ be a fixed real number, and let $\Delta$ be a geodesic triangle in $X$ with perimeter strictly less than $2D_\kappa$. We say that $\Delta$ satisfies the $CAT(\kappa)$ inequality if, for all $x,y\in \Delta$ and the respective comparison points $\overline{x},\overline{y}\in \overline{\Delta}$, it holds that
    \begin{equation}\label{eq:Cat(k)inequality}
        d\big(x,y\big)\leq d_{M_\kappa^2}(\overline{x},\overline{y}).
    \end{equation}
    $X$ is called a $CAT(\kappa)$ space if it is $D_\kappa$-geodesic and each triangle $\Delta$ in $X$ with perimeter strictly less than $2D_\kappa$ satisfies \eqref{eq:Cat(k)inequality}.
\end{definition}

Note that if $\kappa \leq 0$, the inequality \eqref{eq:Cat(k)inequality} must hold for every geodesic triangle in $X$. A first example of a $CAT(\kappa)$ space is the model space $M_\kappa^2$, while a typical example of a $CAT(0)$ space is a pre-Hilbert space. An important characterization result is given in the following.

\begin{theorem}[\cite{bridson2013metric}, Theorem~1.12, p.~165]
Every $CAT(\kappa)$ space is also a $CAT(\kappa')$ space for all $\kappa' \geq \kappa$. Moreover, if $X$ is a $CAT(\kappa')$ space for every $\kappa' > \kappa$, then $X$ is a $CAT(\kappa)$ space.
\end{theorem}
The previous result implies that every $CAT(\kappa)$ space with $\kappa \le 0$ is also a $CAT(0)$ space; in particular, this happens for the hyperbolic space $\mathbb{H}^n$.

We remark that $CAT(\kappa)$ spaces are, in general, not complete with respect to their given metric. A well-known complete $CAT(\kappa)$ space is obtained when \(\kappa=0\). In fact, complete $CAT(0)$ spaces are referred to as Hadamard spaces.

%%%%%%%%%%%%%%%%%%%%%IMP aggiungere agli spazi cat(0) (see~\cite[Proposition~2.2, p.~176]{bridson2013metric}) con labe {ConvConditionCat0}

We can now deal with several notions of convexity in the setting of the geodesic spaces.

As in \cite{bridson2013metric}, we say that a subset $C$ of a geodesic space $X$ is convex if for all $x,y\in C$, the image of every geodesic joining them lies in $C$. If $X$ is uniquely geodesic, $C$ is said to be {strongly convex if it is convex and there is no other local geodesic joining $x$ and $y$. 
The converse implication does not hold in general and a counterexample can be obtained working on the torus endowed with the intrinsic metric.
We recall that a metric space \((X,d)\) satisfies the convexity condition of the metric if for all $x, y, w \in X$ and for all $t \in [0,1]$ it holds:
\begin{equation}
\label{ConvCond}
    d(w, (1 - t)x \oplus ty) \leq (1 - t)d(w,x) + t d(w,y).
\end{equation}

In particular, as stated in \cite[Proposition~2.2, p.~176]{bridson2013metric}, if $(X,d)$ is a CAT(0) space, the distance function $d : X \times X \to \mathbb{R}$ is convex, which means that, given any pair of geodesics $\gamma : [0,1] \to X$ and $\gamma' : [0,1] \to X$, the following inequality holds for all $t \in [0,1]$:
\begin{equation} 
\label{ConvConditionCat0}
d(\gamma(t), \gamma'(t)) \leq (1-t) d(\gamma(0), \gamma'(0)) + t\, d(\gamma(1), \gamma'(1)).
\end{equation} 

For later use, we also recall the notion of convexity for functions on geodesic spaces.

\begin{definition}
\label{def:geodconvf}
Let $(X,d)$ be a geodesic space, let $C \subseteq X$ be a convex subset, and let
$f : C \to (-\infty,+\infty]$.
We say that $f$ is geodesically convex if for every geodesic
$\gamma : [0,1] \to C$ and for all $ t \in [0,1]$ one has
\begin{equation*}
%\label{fgeodconv}
f(\gamma(t))
\leq
(1-t)\,f(\gamma(0))
+
t\,f(\gamma(1)).
\end{equation*}
If $X$ is uniquely geodesic, this is equivalent to $ f\bigl((1-t)x \oplus ty\bigr) \leq (1-t)\,f(x) + t\,f(y)$  for all $x,y \in C$ and for all $ t \in [0,1]$.
\end{definition}

Since the classical notions of convergence in normed spaces do not directly extend to this setting, an alternative notion, introduced by Lim in \cite{lim1976remarks} and known as $\Delta$-convergence, plays a central role in the analysis.

\begin{definition} % [\textbf{\emph{$\Delta$-convergence}}]
\label{defDeltaConvergence}
    Let $\{x_n\}$ be a bounded sequence in a metric space $(X,d)$. First, for all $x\in X$, define the quantity 
    \[
    r(x,\{x_n\})=\limsup\limits_{n\rightarrow\infty}{d(x,x_n)}.
    \]
    Then, for $\{x_n\}$ we can define the asymptotic radius $r(\{x_n\})$ and the asymptotic center $A(\{x_n\})$ respectively as follows:
    \begin{align*}
        r(\{x_n\})&:=\inf\limits_{x\in X} r(x,\{x_n\}), \\
        A(\{x_n\})&:=\{x\in X \,: \, r(x,\{x_n\})= r(\{x_n\})\}.
    \end{align*}
    The sequence ${x_n}$ is said to $\Delta$-converge to $x$, in symbol $x_n\xrightarrow{\Delta}x$, if for every subsequence $\{x_{n_k}\}$ of $\{x_n\}$, the point $x$ is the unique asymptotic center of $\{x_{n_k}\}$, that is equivalent to say that $x$ is the only element to satisfy:
    \begin{equation}\label{eq:weakdeltaconc}
    \limsup_{k\to\infty} d(x_{n_k},x)\le \limsup_{k\to\infty} d(x_{n_k},y) \quad \,\text{for all } y\in X.
    \end{equation}
\end{definition}

%It is clear that $\Delta$-convergence is much weaker than strong convergence in a metric space.
As noted in \cite{lim1976remarks}, the inequality \eqref{eq:weakdeltaconc} does not guarantee the uniqueness of the asymptotic center.
%: in fact, a point is an asymptotic center if and only if it satisfies \eqref{eq:weakdeltaconc}.
However, it is known that if a sequence ${x_n}$ in a complete $CAT(\kappa)$ space has asymptotic radius strictly less than $\frac{D_\kappa}{2}$, then it admits a unique asymptotic center (see \cite[Proposition 7]{dhompongsa2006fixed} for $\kappa=0$ case and \cite[Proposition 4.1]{espinola2009cat} for $\kappa>0$ case).
Clearly, $\Delta$-convergence is much weaker than strong convergence in a metric space.
Moreover, it has been proved in \cite{espinola2009cat} that $\Delta$-convergence coincides with weak convergence in Hilbert spaces.

Finally, we recall a metric condition fundamental for what follows.
\begin{definition} %[\textbf{\emph{Nonexpansive map}}]
Let $(X,d)$ be a geodesic space and $C \subset X$. A mapping $T:C\to C$ is said to be nonexpansive if $d(Tx,Ty) \le d(x,y)$ for all $x,y\in C$.
\end{definition}

The study of nonexpansive mappings in $CAT(\kappa)$ spaces was initiated by Kirk (see, e.g., \cite{kirk2003geodesic,kirk2004geodesic}), and an important class of nonexpansive operators in the linear setting is the so-called forward map, which basically coincides with the step of classical gradient descent algorithm as stated in the following result.
\begin{theorem}[See ~\cite{bauschke2010baillon,baillon1977}]
\label{Baillon-Haddad} 
Let $H$ be a real Hilbert space, and let $f:H\to\mathbb{R}$ be a convex and Fréchet–differentiable function with $L$-Lipschitz continuous gradient $\nabla f$.  
Fix $\eta\in(0,2/L)$ and define the forward operator $F:H\to H$ by
\[
F = I - \eta \nabla f.
\]
Then $F$ is averaged nonexpansive, that is,
\(
F = \lambda I + (1-\lambda)T
\)
for some $\lambda\in(0,1)$ and some nonexpansive mapping $T:H\to H$. Moreover,
\[
\theta^* \in \arg\min_x f(x) 
\quad\Longleftrightarrow\quad 
\theta^* = F(\theta^*)
\quad\Longleftrightarrow\quad 
\theta^* = T(\theta^*),
\]
so that the fixed point set $\mathrm{Fix}(T)$ coincides with the set of minimizers of $f$.
\end{theorem}

The nonexpansivity of this operator played a crucial role in neural network hyperparameter optimization and, in particular, was essential for the development of the HalpernSGD optimizer (see \cite{foglia2024halpernsgd,colao2025optimizer,colao2026convergence}).

In the nonlinear setting, more precisely on Hadamard manifolds, the Euclidean nonexpansivity conclusion for the forward operator does not extend automatically. Section~\ref{sec:iperbolicHalpernGD} provides an explicit counterexample on \(\mathbb H^2\).

However, an important example of a nonexpansive mapping in Hadamard spaces is the proximal mapping (or resolvent), defined as follows.
\begin{definition}%[\textbf{\emph{Proximal mapping in Hadamard spaces}}]
\label{def:proxHadamard}
Let $(X,d)$ be a Hadamard space and let $f:X\to(-\infty,+\infty]$ be a proper, lower semicontinuous, and convex function.  
For any $\lambda>0$, the proximal mapping (or resolvent) of $f$ (see~\cite[Ch.~2.2 and Ch.~4.2, Lemma~4.2.2]{bacak2014convex}) is defined by
\[
J_{\lambda f}(x)
:= \operatorname*{argmin}_{y\in X}\Big\{\, f(y) + \tfrac{1}{2\lambda}\, d(x,y)^2 \Big\}, \qquad x\in X.
\]
\end{definition}
This operator can be used to develop an analogue of the HalpernGD optimizer in the hyperbolic setting, as discussed in Section~\ref{sec:iperbolicHalpernGD}.

In \cite{kirk2008concept}, the concept of $\Delta$-convergence was investigated for nonexpansive mappings in $CAT(0)$ spaces, showing that such spaces provide a natural framework for this type of convergence. Moreover, by \cite[Theorem 18, p. 207]{kirk2003geodesic}, every nonexpansive map defined on a nonempty, closed, bounded and convex subset of a complete $CAT(0)$ space admits a fixed point.
These results justify the choice of the setting for our work.

%%%%%%%%%%%%%%%%%%%%%%%%%%%%%%%%%%%%%%%%%%%%%%%%%%%%%%%%%%%%%%%%%%%%%%%%%%%%%%%%%%%%%%%%%%%%%%%%%%%%%%%%%%%%%%%%%%%%%%%%%%%%%%%%%%%%%%%%%%%%%%%%%%%%%%%%%%%%%%%%%%%%%%%%%%%%%%%%%%%%%%%%%%%%%%%%%%%%%%%%%%%%%%%%%%%%%%%%%%%%%%%%%%%%%%%%%%%%%%%%%%%%%%%%%%%%%%%%%%%%%

\section{Krasnosel'skii-Mann and Halpern iterative methods}
\label{sec:statoarte}
Iterative methods for approximating fixed points of mappings have a long and rich history, originating in the context of Banach spaces and progressively extending to more general nonlinear spaces such as CAT(0) spaces and geodesic spaces.
Their study has been a flourishing area of research over the past decades, with numerous applications in nonlinear analysis, optimization, and partial differential equations \cite{browder1965, baillon1996, rockafellar1970, rockafellar1976}.
The Krasnosel'ski\u{\i}--Mann (KM) and Halpern-type iterations stand out as the most classical and extensively investigated schemes. 
The study of their convergence behavior has given rise to a vast literature. 
Beyond the linear setting, however, the intrinsic convex structure of these algorithms makes it more challenging to establish approximation results for fixed points. 
Consequently, substantial progress has been achieved only in certain specific classes of metric spaces, such as hyperbolic spaces, the Hilbert ball, and Hadamard manifolds.
\\

Let \(X\) be a Banach space, let \(K\subseteq X\) be nonempty and convex, and let
\(T:K\to K\) be a nonexpansive mapping. The Krasnosel'skii--Mann iteration is
defined by
\begin{equation}\label{eq:Krasnoselskii-Mann iteration}
x_{n+1}=(1-\lambda_{n+1})x_n+\lambda_{n+1}Tx_n, \qquad x_0\in K,
\end{equation}
where $(\lambda_n)_{n \ge 1} \subset [0,1)$. Its convergence theory in Banach spaces is classical:
asymptotic regularity and convergence depend on both the geometry of the space and
the control sequence, while strong convergence can be recovered under additional
compactness assumptions; see, e.g., \cite{mann1953mean,ishikawa1976}.

In a uniquely geodesic space, the natural nonlinear analogue is
\begin{equation}\label{eq:KM-Geodesic}
x_{n+1}=(1-\lambda_{n+1})x_n\oplus \lambda_{n+1}Tx_n, \qquad x_0\in K,
\end{equation}
where \(\oplus\) denotes the convex combination along the unique geodesic joining
\(x_n\) and \(Tx_n\), meaning that $x_{n+1}$ lies on the geodesic segment $[x_{n}, Tx_{n}]$ (see notation in \eqref{oplus}). If \(X\) is a complete \(CAT(0)\) space and \(\mathrm{Fix}(T)\neq\emptyset\),
then \eqref{eq:KM-Geodesic} \(\Delta\)-converges to a fixed point of \(T\) under standard
assumptions on \((\lambda_n)\), for instance \(\sup_n\lambda_n<1\) and
\(\sum_n\lambda_n=+\infty\) \cite{dhompongsa2008convergence}. More generally,
analogous results hold in complete \(CAT(\kappa)\) spaces under the radius
restriction $d(x_0,\mathrm{Fix}(T))\leq \tfrac{D_\kappa}{4}$ and the condition \(\sum_n \lambda_n(1-\lambda_n)=+\infty\); if the space
is boundedly compact, the convergence is strong \cite{he2012mann}. In the special
case \(\kappa=0\), the radius restriction disappears.
% Since in this paper we focus on asymptotic regularity in the \(CAT(0)\) setting, iteration \eqref{eq:KM-Geodesic} will be the formulation used throughout Section~4.

 We turn now our attention to a closely related scheme which, under suitable assumptions, ensures strong convergence.
Let \(X\) be a Banach space, let \(K\subseteq X\) be nonempty, closed and convex, and let \(T:K\to K\) be nonexpansive with \(\mathrm{Fix}(T)\neq\emptyset\). For a sequence $(\alpha_k)_{k \ge 1} \subset (0,1]$, the Halpern iteration is defined by
\begin{equation}\label{eq:Halpern methods}
    z_{n+1}=\alpha_{n+1}u+(1-\alpha_{n+1})Tz_n,
    \qquad z_0,u\in K.
\end{equation}
Introduced by Halpern~\cite{halpern1967}, this scheme is a classical strongly
convergent alternative to the Krasnosel'skii--Mann iteration.
In Hilbert spaces, strong convergence is known under the standard conditions
\begin{itemize}
    \item[\textbf{(C1)}] $\lim_{n \to \infty} \alpha_n = 0\qquad$ \textbf{(C2)}  $\sum_{n=0}^\infty \alpha_n = +\infty$.
\end{itemize}
together with a mild regularity assumption on the control sequence, for instance,
\begin{itemize}
    \item [\textbf{(C3)}] $\sum_{n=0}^{\infty} \vert \alpha_{n+1}-\alpha_n \vert < \infty$.
\end{itemize}
(see, e.g.,~\cite{halpern1967,wittmann1992approximation}). More generally, the viscosity iteration is given by
\begin{equation}\label{eq:Viscosity iteration}
    y_{n+1}=\alpha_{n+1}f(y_n)+(1-\alpha_{n+1})Ty_n,
    \qquad y_0\in K,
\end{equation}
where \(f:K\to K\) is a \(\beta\)-contraction. Choosing \(f\equiv u\), one recovers
the Halpern scheme~\eqref{eq:Halpern methods}~\cite{xu2004viscosity}.

In a complete uniquely geodesic space, the nonlinear counterpart of
\eqref{eq:Halpern methods} is
\begin{equation}\label{eq:HalpernGeodesic}
    z_{n+1}=\alpha_{n+1}u\oplus(1-\alpha_{n+1})Tz_n,
    \qquad z_0,u\in K.
\end{equation}
If \(X\) is a complete \(CAT(0)\) space and \((\alpha_n)\) satisfies the above
conditions (C1), (C2) and (C3) then \eqref{eq:HalpernGeodesic} converges strongly to the metric projection of \(u\) onto \(\mathrm{Fix}(T)\)~\cite{saejung2009halpern}. This strong convergence is one of the main advantages of Halpern-type methods in the nonlinear setting.

%Since our main result in Section~4.2 concerns the viscosity form in complete \(CAT(0)\) spaces, we will use this framework from now on.

%%%%%%%%%%%%%%%%%%%%%%%%%%%%%%%%%%%%%%%%%%%%%%%%%%%%%%%%%%%%%%%%%%%%%%%%%%%%%%%%%%%%%%%%%%%%%
\subsection{Rate of asymptotic regularity}

Having examined the main convergence results for the KM iterative method, we now recall some important results on its rate of asymptotic regularity. This analysis is relevant to the use of such iterations in optimization algorithms for training deep learning models.

Consider now a nonexpansive operator \(T : K \to K\), where \(K\) is a bounded and convex subset of a normed space \(X\).

In~\cite{baillon1996}, it was shown that the Krasnosel'ski\u{\i}--Mann scheme \eqref{eq:Krasnoselskii-Mann iteration} yields an asymptotic regularity rate that is independent of both the initial point \(x_0\) and the particular operator \(T\), depending instead on the diameter of the set \(K\).
This was recently confirmed by Cominetti in~\cite{cominetti2014}, who proved that
\[
\|x_n-Tx_n\| \le \frac{\operatorname{diam}(K)}{\sqrt{\pi \sum_{k=1}^{n}\lambda_k(1-\lambda_k)}}.
\]

Moreover, Bravo and Cominetti later demonstrated that this bound is sharp~\cite{bravo2018}. It follows that, under the standard assumptions on $(\lambda_k)$, the convergence rate remains bounded by \(\mathcal{O}(1/\sqrt{n})\).

An explicit example of a nonexpansive linear operator is the right-shift operator acting on \(\bigl(\ell^1(\mathbb{N}), \|\cdot\|_1\bigr)\). For this operator, the KM iteration starting from \(x^0=(1,0,0,\dots)\) satisfies, for all \(n \ge 0\), \(\|x^n - T x^n\|_1 \ge \frac{1}{\sqrt{n+1}}\) (see \cite{contreras2023}).
 
However, faster rates can be achieved by alternative schemes: for instance, in Hilbert spaces, in~\cite{bot2022} a convergence rate of $\mathcal{O}(1/n)$ was obtained via a Nesterov-type acceleration technique.

As in the case of the Krasnosel'ski\u{\i}--Mann iteration, we review some results from the literature on the rate of asymptotic regularity of the viscosity iteration \eqref{eq:Viscosity iteration} and the Halpern iteration \eqref{eq:Halpern methods}. When applied to convex bilevel optimization problems, the authors of~\cite{sabach2017first} derived the following estimate in the Euclidean setting \(X = \mathbb{R}^d\):

\begin{equation*}
    \|  z_n-Tz_n \| \leq \frac{4}{n+1} \| z_0 - P_{\mathrm{Fix}(T)}u \|,
\end{equation*}
where \(P_{\mathrm{Fix}(T)}(\cdot)\) denotes the metric projection onto the fixed-point set \(\mathrm{Fix}(T)\).  
Later, Contreras et al.~\cite{contreras2023} showed that this bound extends to general normed linear spaces.
In the specific case where \(T\) is a nonexpansive operator defined on a Hilbert space, sharper results are available. 
Lieder~\cite{lieder2021convergence} established the improved bound
\begin{equation*}
    \| z_n - Tz_n \| \leq \frac{2}{n+1}\, \| z_0 - P_{\mathrm{Fix}(T)}u \|,
\end{equation*}
and furthermore constructed a family of nonexpansive operators on \(H=\mathbb{R}\) for which this inequality holds with equality, thus demonstrating the optimality of the rate. A simple example, highlighting the sharpness of the bound, is given by the rotation operator on the plane $\mathbb{R}^2$ with angle $\theta_n = \frac{\pi}{n+1}$. 
For $\lambda _k = \frac{k}{k+1}$ and starting point $x_0=(1,0)$, one obtains $\|x_n - T x_n\|_2 = \frac{2}{n+1}$ for every $n$ (see \cite{contreras2023}).
Regarding the asymptotic regularity for both KM and Halpern schemes, the rates known in linear spaces extend to the nonlinear setting of geodesic spaces. 
In particular, Leu\c{s}tean and Pinto~\cite{leustean2022alternating} 
established that the alternating Halpern--Mann iteration in complete $CAT(0)$ and more general hyperbolic spaces satisfies an asymptotic regularity bound of order $O(1/\sqrt{n})$, thus confirming the rate already known for the Krasnosel'ski\u{\i}--Mann in Banach spaces. 
On the other hand, Pinto and Pischke~\cite{pinto2025halpern} proved that the Halpern iteration  with a variable anchor point in Hadamard spaces achieves the optimal rate $O(1/n)$, in agreement with the estimates derived in Hilbert spaces. 
These results confirm that the geometry of nonpositively curved metric spaces preserves the asymptotic efficiency of the classical fixed-point algorithms.
\\

\section{Main result}
\label{sec:ContributionIperbolici}

This section outlines our main research contributions. 
We first extend the classical analysis of Cominetti~\cite{cominetti2014} on the asymptotic regularity rate of the Krasnosel'ski\u{\i}--Mann (KM) iteration to the nonlinear setting of complete $CAT(0)$ spaces, establishing the same convergence rate as in Banach spaces (see Subsection~\ref{CONTRIBUTION1_ip}).

We then consider the viscosity iteration, and by adapting the argument of Sabach et al.~\cite{sabach2017first}, we derive an asymptotic regularity bound in $CAT(0)$ spaces. As a special case, this result also covers the Halpern iteration, thus recovering in the nonlinear framework, via a proof adapted from the linear setting, estimates already known in the literature.

While the general argument for the KM scheme requires the space $(X,d)$ to be complete, uniquely geodesic, and to satisfy the convexity condition \eqref{ConvCond} for the metric $d$, we will assume $X$ to be a complete $CAT(0)$ space. This choice reflects the central role that Hadamard spaces play in the convergence analysis of fixed-point iterations in the literature. Recall that every $CAT(0)$ space is uniquely geodesic and satisfies the convexity condition (CC) (see~\cite[Lemma~2.4]{dhompongsa2008convergence}).

%%%%%%%%%%%%%%%%%%%%%%%%%%%%%%%%%%%%%%%%%%%%%%%%%%%%%%%%%%%%%%%%%%%%%%%%%%%%%%%%%%%%%%%%%%%%%%%%%%%%%%%%%%%%%%%%%%%%%%%%%%%%%%%%%%%%%%%%%%%%%%%%%%%%%%%%%%%%%%%%%%%%%%%%%%%%%%%%%%%%%

\subsection{Asymptotic Regularity of the Krasnosel'ski\u{\i}--Mann Iteration in Complete $CAT(0)$ Spaces}
\label{CONTRIBUTION1_ip}
We begin this subsection by stating our main result through the following theorem.

\begin{theorem}\label{thm:KM-rate}
Let $(X,d)$ be a complete $CAT(0)$ space, and let $K \subset X$ be nonempty, closed, convex, and bounded. Suppose that $T:K \to K$ is nonexpansive. Then, for the Krasnosel'ski\u{\i}--Mann iteration \eqref{eq:KM-Geodesic} generated by a sequence $\{\lambda_n\} \subset (0,1)$, the following inequality holds
\[ d(x_n, Tx_n)\ \le\ \frac{\mathrm{diam}(K)}{\sqrt{\pi\, \sum_{i=1}^{n} \lambda_i(1 - \lambda_i)}}\,,\qquad n \ge 1. \]
\end{theorem}

%\begin{theorem}[Rate of asymptotic regularity of KM iteration in complete $CAT(0)$ spaces]\label{thm:KM-rate} Let $(X,d)$ be a complete $CAT(0)$ space and let $K\subset X$ be nonempty, closed, convex and bounded, with $T:K\to K$ nonexpansive. For the Krasnosel'ski\u{\i}--Mann iteration \eqref{eq:KM-Geodesic} generated by a sequence $(\lambda_n)\subset(0,1)$, one has
%\[d(x_n,Tx_n)\ \le\ \frac{\mathrm{diam}(K)}{\sqrt{\pi\ \sum_{i=1}^{n}\lambda_i(1-\lambda_i)}}\,,\qquad n\ge 1.\]
%\end{theorem}

The proof of Theorem~\ref{thm:KM-rate} is based on a sequence of auxiliary lemmas that extend the argument of Cominetti~\cite{cominetti2014} to the nonlinear metric setting. From this point on, we assume that $(X,d)$ and $T$ satisfy the hypotheses of Theorem~\ref{thm:KM-rate}, and that $\mathrm{diam}(K) = 1$.

\begin{lemma}\label{prop:representation}
Let $(x_m)$ be the Krasnosel'ski\u{\i}--Mann iterates defined by \eqref{eq:KM-Geodesic}, and define the weights
\[
\pi^n_k := \lambda_k \prod_{j=k+1}^n (1-\lambda_j),\]
with the conventions $\lambda_0 := 1$ and $Tx_{-1} := x_0.$
Then, for all $m \geq 0$ and $q \in X$, the following inequality holds:
\begin{equation}\label{eq:representation}
d(x_m, q) \leq \sum_{j=0}^m \pi^m_j\, d(Tx_{j-1}, q).
\end{equation}
\end{lemma}

\begin{proof}
The proof proceeds by induction on $m$.

\noindent
For $m = 0$, the claim follows immediately:
\[
d(x_0, q) = \lambda_0\, d(Tx_{-1}, q) 
= \pi^0_0\, d(Tx_{-1}, q)
= \sum_{j=0}^0 \pi^0_j\, d(Tx_{j-1}, q).
\]

\noindent
Assume that \eqref{eq:representation} holds for $m - 1$. By the definition of $x_m$ we have:
\[
d(x_m, q) = d\big((1 - \lambda_m)x_{m-1} \oplus \lambda_m Tx_{m-1},\, q\big).
\]
Applying the convexity condition of the metric (CC):
\[
d(x_m, q) \leq (1 - \lambda_m)\, d(x_{m-1}, q) + \lambda_m\, d(Tx_{m-1}, q).
\]
Using the inductive hypothesis on $d(x_{m-1}, q)$, we obtain:
\[
d(x_m, q) \leq (1 - \lambda_m) \sum_{j=0}^{m-1} \pi^{m-1}_j\, d(Tx_{j-1}, q) + \lambda_m\, d(Tx_{m-1}, q),
\]
and this coincides with:
\[
d(x_m, q) \leq \sum_{j=0}^m \pi^m_j\, d(Tx_{j-1}, q),
\]
as claimed.
\end{proof}

\begin{lemma}\label{prop:step-sum}
Let $(x_n)$ be the Krasnosel'ski\u{\i}--Mann iterates defined by \eqref{eq:KM-Geodesic}. Then, for all $m \geq 0$ and $r \geq 1$, the following inequality holds:
\[
d(x_{m+r}, x_m)\ \le\ \sum_{k=m+1}^{m+r} \pi^{m+r}_{k}\, d(Tx_{k-1}, x_m).
\]
\end{lemma}

\begin{proof}
We proceed by induction on $r$.

\noindent
For $r = 1$, using the definition of $x_{m+1}$ by (\ref{eq:KM-Geodesic}), the convexity condition of the metric (CC) and the definition of $\pi^n_k$ , we have:
\[
\begin{aligned}
d(x_{m+1}, x_m)
&= d\big((1 - \lambda_{m+1})x_m \oplus \lambda_{m+1} Tx_m,\, x_m\big) \\
&\le (1 - \lambda_{m+1})\, d(x_m, x_m) + \lambda_{m+1}\, d(Tx_m, x_m) \\
&= \sum_{k=m+1}^{m+1} \pi^{m+1}_{k}\, d(Tx_{k-1}, x_m).
\end{aligned}
\]

\noindent
Now assume the inequality holds for some $r \geq 1$. Then, using again the iterates definition and the CC of the metric,
\[
\begin{aligned}
d(x_{m+r+1}, x_m)
&= d\big((1 - \lambda_{m+r+1})x_{m+r} \oplus \lambda_{m+r+1} Tx_{m+r},\, x_m\big) \\
&\le (1 - \lambda_{m+r+1})\, d(x_{m+r}, x_m) + \lambda_{m+r+1}\, d(Tx_{m+r}, x_m).
\end{aligned}
\]
Applying the inductive hypothesis to the first term and considering that $\pi^{m+r+1}_k = (1 - \lambda_{m+r+1})\, \pi^{m+r}_k$ for $m+1 \le k \le m+r$  and $\pi^{m+r+1}_{m+r+1} = \lambda_{m+r+1}$ we obtain:
\[
\begin{aligned}
d(x_{m+r+1}, x_m)
&\le \sum_{k=m+1}^{m+r} \pi^{m+r+1}_k\, d(Tx_{k-1}, x_m) + \pi^{m+r+1}_{m+r+1}\, d(Tx_{m+r}, x_m) \\
&= \sum_{k=m+1}^{m+r+1} \pi^{m+r+1}_k\, d(Tx_{k-1}, x_m),
\end{aligned}
\]
which proves the statement for $r + 1$. The claim follows by induction.
\end{proof}

\begin{proposition}\label{prop:two-index}
Let $(x_n)$ be the Krasnosel'ski\u{\i}--Mann iterates defined by \eqref{eq:KM-Geodesic}, and let $n = m + r$ with $m \geq 0$, $r \geq 1$ (i.e., $0 \leq m < n$). Then the following estimate holds:
\begin{equation}\label{eq:two-index}
d(x_n, x_m) \ \le\ \sum_{j=0}^{m} \sum_{k=m+1}^{n} \pi^n_k\, \pi^m_j\, d\big(Tx_{j-1},\, Tx_{k-1}\big).
\end{equation}
\end{proposition}

\begin{proof}
By Lemma~\ref{prop:step-sum} and using $n = m + r$, we have:
\[
d(x_n, x_m) = d(x_{m+r}, x_m) \le \sum_{k=m+1}^{m+r} \pi^{m+r}_k\, d(Tx_{k-1}, x_m).
\]
Fix $k$ in the summation and apply Lemma~\ref{prop:representation} with $q = Tx_{k-1}$ and index $m$:
\[
d(Tx_{k-1}, x_m) \le \sum_{j=0}^m \pi^m_j\, d(Tx_{j-1}, Tx_{k-1}).
\]
Substituting this into the previous estimate and observing that $n = m + r$, we obtain:
\[
d(x_n, x_m) \le \sum_{k=m+1}^{n} \pi^n_k \sum_{j=0}^m \pi^m_j\, d(Tx_{j-1}, Tx_{k-1}),
\]
which coincides with the desired inequality.
\end{proof}

\begin{proposition}\label{prop:cmn}
If for all $0\le m<n$, we define $c_{m,n}$ recursively by
\[
c_{m,n}\ :=\ \sum_{j=0}^{m}\ \sum_{k=m+1}^{n}\ \pi^m_j\,\pi^n_k\,c_{j-1,k-1},
\quad\text{with } \ c_{-1,n}:=1 \ \text{ for all } n\ge 0.
\]
Then, for all $0\le m<n$,
\begin{equation}\label{eq:cmn-bound}
d(x_n,x_m)\ \le\ c_{m,n}.
\end{equation}
\end{proposition}

\begin{proof}
We proceed by induction on $n$. Fix $n>m\ge 0$ and assume that
\[
d(x_{k},x_{j})\ \le\ c_{j,k}\qquad\text{for all } 0\le j<k\le n-1.
\]
By Proposition~\ref{prop:two-index} (i.e., \eqref{eq:two-index}),
\[
d(x_n,x_m)\ \le\ \sum_{j=0}^{m}\ \sum_{k=m+1}^{n}\ \pi^m_j\,\pi^n_k\ d\big(Tx_{j-1},\,Tx_{k-1}\big).
\]
We split the double sum to treat separately the case $j=0$ and the case $1\le j\le m$.
\[
\sum_{j=0}^{m}\ \sum_{k=m+1}^{n}\ \pi^m_j\,\pi^n_k\ d\big(Tx_{j-1},\,Tx_{k-1}\big) 
\]
\[
= \sum_{k=m+1}^{n}\ \pi^m_0\,\pi^n_k\ d\big(Tx_{-1},\,Tx_{k-1}\big) 
+\sum_{j=1}^{m}  \sum_{k=m+1}^{n}\ \pi^m_j\,\pi^n_k\ d\big(Tx_{j-1},\,Tx_{k-1}\big).
\]

Recalling $Tx_{-1}=x_0$, $c_{-1,n}:=1$ and using that $d\big(x_0,\,Tx_{k-1}\big) \leq diam (K) =1 $ we obtain for the first term:
\[
\sum_{k=m+1}^{n}\ \pi^m_0\,\pi^n_k\ d\big(x_0,\,Tx_{k-1}\big) \leq \sum_{k= m+1}^{n}\ \pi^m_0\,\pi^n_k\ c_{-1,k-1}.
\]

As for the second term, using the nonexpansiveness of $T$ and the induction hypothesis,
\[
d\big(Tx_{j-1},Tx_{k-1}\big)\ \le\ d(x_{j-1},x_{k-1})\ \le\ c_{j-1,k-1}.
\]
We then obtain:
\[
\sum_{j=1}^{m}\ \sum_{k=m+1}^{n}\ \pi^m_j\,\pi^n_k\ d\big(Tx_{j-1},\,Tx_{k-1}\big) \leq\sum_{j=1}^{m} \sum_{k=m+1}^{n}\ \pi^m_j\,\pi^n_k\ c_{j-1,k-1}
\]

Summing both terms yields
\[
d(x_n,x_m)\ \le\ \sum_{k=m+1}^{n}\ \pi^m_0\,\pi^n_k\ c_{-1,k-1} +\sum_{j=1}^{m} \sum_{k=m+1}^{n}\ \pi^m_j\,\pi^n_k\,c_{j-1,k-1}
\ =\ c_{m,n},
\]
which proves \eqref{eq:cmn-bound} and completes the induction.
\end{proof}

\begin{corollary}\label{cor:pn}
For every $n\ge 0$, the Krasnosel'ski\u{\i}--Mann iterates satisfy
\[
d(x_n,Tx_n) \ \le \ \frac{c_{n,n+1}}{\lambda_{n+1}} \ =:\ P_n.
\]
\end{corollary}

\begin{proof}
By definition, we have \(d(x_n,x_{n+1}) \ = \ \lambda_{n+1}\, d(x_n,Tx_n)\), and by Proposition~\ref{prop:cmn} it holds \( d(x_n,x_{n+1}) \le c_{n,n+1}\).
Dividing both sides by $\lambda_{n+1}$ completes the proof.
\end{proof}

\begin{remark}\label{rem:probabilistic}
Using the probabilistic interpretation of the coefficients $\pi^n_k$ (see~\cite{cominetti2014}), one obtains the estimate:
\[
\sqrt{\;\sum_{i=1}^n \lambda_i(1-\lambda_i)} \;\, P_n \ \le \ \frac{1}{\sqrt{\pi}}.
\]

We observe that the construction of the weights $\pi^n_k$ is purely algebraic and does not depend on the geometry of the underlying space. As a consequence, the probabilistic estimate from~\cite{cominetti2014} extends directly to the $CAT(0)$ setting.

Combining this inequality with Corollary~\ref{cor:pn} immediately yields the bound stated in Theorem~\ref{thm:KM-rate}.
\end{remark}

%%%%%%%%%%%%%%%%%%%%%%%%%%%%%%%%%%%%%%%%%%%%%%%%%%%%%%%%%%%%%%%%%%%%%%%%%%%%%%%%%%%%%%%%%%%%%%%%%%%%%%%%%%%%%%%%%%%%%%%%%%%%%%%%%%%%%%%%%%%%%%%%%%%%%%%%%%%%%%%%%%%%%%%%%%%%%%%%%%

\subsection{Halpern/Viscosity Iteration in Complete $CAT(0)$ Spaces: an \(O(1/k)\) Asymptotic Regularity Bound}

Throughout this subsection, let $(X,d)$ be a complete $CAT(0)$ space, and let $K \subset X$ be nonempty, closed, and convex. 
Assume that $T: K \to K$ is nonexpansive with $\text{Fix}(T) \neq \emptyset$, and that $f: K \to K$ is a $\beta$-contraction for some $\beta \in [0,1)$.

Given a sequence $(\alpha_k)_{k \ge 1} \subset (0,1]$, we consider the viscosity-type Halpern iteration defined by
\begin{equation}\label{eq:viscHalpern}
    x_{k+1} = \alpha_{k+1}\, f(x_k) \oplus (1 - \alpha_{k+1})\, T(x_k), \qquad x_0 \in K
\end{equation}
which is the nonlinear counterpart of \eqref{eq:Viscosity iteration} and for \(f\equiv u\) it coincides with iteration \eqref{eq:HalpernGeodesic}.
\begin{lemma}\label{lem:bounded}
Set \(\bar{x} \in \text{Fix}(T)\), and set \(z := f(\bar{x})\). Then the sequence \((x_k)\) generated by \eqref{eq:viscHalpern} is bounded and, for all \(k \ge 0\),
\[
d(x_k, \bar{x}) \ \le\ C_{\bar{x}}
\quad\text{with}\quad
C_{\bar{x}} := \max\!\Big\{\,d(x_0, \bar{x}),\ \tfrac{1}{1 - \beta}\, d(z, \bar{x}) \Big\}.
\]
\end{lemma}

\begin{proof}
Using the metric convexity in $CAT(0)$ and the nonexpansiveness of \(T\), we have:
\begin{align*}
d(x_{k+1}, \bar{x}) 
&\le\ (1 - \alpha_{k+1})\, d(Tx_k, \bar{x}) + \alpha_{k+1}\, d\big(f(x_k), \bar{x}\big)\\
&\le\ (1 - \alpha_{k+1})\, d(x_k, \bar{x}) + \alpha_{k+1} \big(\beta\, d(x_k, \bar{x}) + d(z, \bar{x})\big),
\end{align*}
hence \(d(x_{k+1}, \bar{x}) \le d(x_k, \bar{x}) + \alpha_{k+1}\, d(z, \bar{x}) - \alpha_{k+1}(1 - \beta)\, d(x_k, \bar{x})\).

If \(d(x_k, \bar{x}) \ge \tfrac{1}{1 - \beta} d(z, \bar{x})\), the last term is nonpositive and \(d(x_{k+1}, \bar{x}) \le d(x_k, \bar{x})\); otherwise, \(d(x_{k+1}, \bar{x}) \le \tfrac{1}{1 - \beta} d(z, \bar{x})\). By induction, the claim follows.
\end{proof}

\begin{lemma}
(see~\cite{sabach2017first})
\label{lem:summation}
Let \(M>0\) and \(\gamma\in(0,1]\).
Assume that \((a_k)\subset[0,\infty)\) satisfies \(a_1\le M\) and
\[
a_{k+1} \le (1-\gamma b_{k+1})\,a_k + (b_k-b_{k+1})\,c_k,\qquad k\ge 1,
\]
where
\[
b_k := \min\!\left\{\frac{2}{\gamma k},\,1\right\}
\quad\text{and}\quad
0\le c_k \le M <\infty \ \ \text{for all } k.
\]
Then, for all \(k\ge 1\),
\[
a_k \le \frac{M\,J}{\gamma k},
\qquad
J:=\left\lceil\frac{2}{\gamma}\right\rceil.
\]
\end{lemma}

\begin{remark}
By the nonexpansiveness of \(T\) and the contraction property of \(f\), it follows that for every \(k \ge 1\) and \(\bar{x} \in \text{Fix}(T)\), defining \(z := f(\bar{x})\), the following inequalities hold:
\[
d\bigl(Tx_{k-1}, \bar{x} \bigr) \le d\bigl(x_{k-1}, \bar{x} \bigr),
\qquad
d\bigl(f(x_{k-1}), z \bigr) \le \beta\, d\bigl(x_{k-1}, \bar{x} \bigr).
\]
\end{remark}

\begin{theorem}\label{thm:rateViscIP}
%Assume that \(\alpha_k := \min\!\big\{ \tfrac{2}{(1-\beta)k},\,1 \big\}\), and let \(C_{\bar{x}}\) be the constant defined in Lemma~\ref{lem:bounded}. Then both the convergence rate and the rate of asymptotic regularity are of order \(O(1/k)\). 
Assume that \(\alpha_k := \min\!\big\{ \tfrac{2}{(1-\beta)k},\,1 \big\}\), and let 
$C_{\bar{x}}$ and $J$ be the constants defined in Lemmas \ref{lem:bounded} and \ref{lem:summation}, respectively. Then the successive differences \(d(x_k,x_{k-1})\) and the asymptotic regularity residuals \(d(Tx_{k-1},x_{k-1})\) admit bounds of order \(O(1/k)\).

In particular, the following estimates hold:
\[
d(x_k, x_{k-1}) \ \le\ \frac{2J\, C_{\bar{x}}}{(1-\beta)\,k}
\quad \text{and} \quad 
d(Tx_{k-1}, x_{k-1}) \ \le\ \frac{2\, C_{\bar{x}}(J+2)}{(1-\beta)\,k},
\qquad (k \ge 1).
\]
\end{theorem}

\begin{proof}

By the triangle inequality, and recalling the definitions of the iterates \(x_k\) and \(x_{k+1}\), we obtain:
\[
\begin{aligned} 
    d(x_{k+1},x_k) &\le d\bigl(x_{k+1},\alpha_{k+1}f(x_{k-1})\oplus(1-\alpha_{k+1})T x_{k-1}\bigr)\\
    & +d\bigl(\alpha_{k+1}f(x_{k-1})\oplus(1-\alpha_{k+1})T x_{k-1},x_k\bigr)\\
    \\
    &=d\bigl(\alpha_{k+1}f(x_k)\oplus(1-\alpha_{k+1})T x_k,\alpha_{k+1}f(x_{k-1})\oplus(1-\alpha_{k+1})T x_{k-1}\bigr) \qquad\qquad(I) \\
    &+d\bigl(\alpha_{k+1}f(x_{k-1})\oplus(1-\alpha_{k+1})T x_{k-1},\;\alpha_k f(x_{k-1})\oplus(1-\alpha_k)T x_{k-1}\bigr)\qquad\qquad(II) 
\end{aligned}
\]
 
Using first the convexity inequality \eqref{ConvConditionCat0} for $CAT(0)$ spaces, and then the nonexpansiveness of $T$ together with the $\beta$-contraction of $f$, we obtain for term \((\mathrm{I})\):
\[
\begin{aligned}
(I)
&\le (1-\alpha_{k+1})\,d(Tx_k,Tx_{k-1})
   + \alpha_{k+1}\,d(f(x_k),f(x_{k-1})) \\
&\le (1-\alpha_{k+1}+\alpha_{k+1}\beta)\,d(x_k,x_{k-1}).
\end{aligned}
\]
For term \((\mathrm{II})\), using the convexity condition of the metric (CC) and the geodesic property \eqref{AppartenenzaAllaGeodetica} (both points lie on the segment \([f(x_{k-1}), T x_{k-1}]\)), we have:
\[
(II) \leq (\alpha_k-\alpha_{k+1})\,d\bigl(f(x_{k-1}),\,T x_{k-1}\bigr).
\]

Substituting the bounds for \((\mathrm{I})\) and \((\mathrm{II})\) into the previous estimate, we obtain

\begin{equation}\label{eq:key-recursion}
\begin{aligned}
d(x_{k+1},x_k)
&\leq (1-\alpha_{k+1}(1-\beta))\,d(x_k,x_{k-1}) \\
&\quad +(\alpha_k-\alpha_{k+1})\,d\bigl(f(x_{k-1}),T x_{k-1}\bigr).
\end{aligned}
\end{equation}
Now let \(\bar{x} \in \text{Fix}(T)\) and set \(z := f(\bar{x})\). 
Applying first the triangle inequality, and then the \(\beta\)-contraction property of \(f\) and the nonexpansiveness of \(T\), we have
\[
\begin{aligned}
d\bigl(f(x_{k-1}),\,T x_{k-1}\bigr)&\le d\bigl(f(x_{k-1}),z\bigr)\;+\;d\bigl(z,\bar x\bigr)\;+\;d\bigl(\bar x,Tx_{k-1}\bigr)\\
&=\;d\bigl(f(x_{k-1}),f(\bar x)\bigr)\;+\;d\bigl(f(\bar x),\bar x\bigr)\;+\;d\bigl(T\bar x,Tx_{k-1}\bigr)\\
&\le \;\beta\,d(x_{k-1},\bar x)\;+\;d\bigl(f(\bar x),\bar x\bigr)\;+\;d(\bar x,x_{k-1})\\
\end{aligned}\]

Finally, by Lemma~\ref{lem:bounded} it holds \( d(x_{k-1}, \bar{x}) \le C_{\bar{x}}\) and 
\(d\bigl(f(\bar{x}), \bar{x}\bigr) \le (1-\beta)\, C_{\bar{x}}\), hence we obtain
\[
d\bigl(f(x_{k-1}),\,T x_{k-1}\bigr)\le\;\beta\,C_{\bar x}\;+\;(1-\beta)\,C_{\bar x}\;+\;C_{\bar x}
\;=\;2\,C_{\bar x}.
\]

By setting $a_k := d(x_k, x_{k-1})$, $b_k := \alpha_k$, $\gamma := 1-\beta$, $ M := 2 C_{\bar x}$ and $c_k := d\bigl(f(x_{k-1}), T x_{k-1}\bigr)$,
and substituting into \eqref{eq:key-recursion}, we obtain
\[ a_{k+1} \le (1-\gamma b_{k+1})\, a_k \;+\; (b_k - b_{k+1})\, c_k, \qquad (k \ge 1), \quad c_k \le M.\]
Moreover, by Lemma~\ref{lem:bounded},
\[
a_1=d(x_1,x_0)\le d(x_1,\bar x)+d(x_0,\bar x)\le 2C_{\bar x}=M.
\]
Therefore, by Lemma~\ref{lem:summation},
\[ d(x_k, x_{k-1}) \ \le\ \frac{2\, J\, C_{\bar x}}{(1-\beta)\, k}\,.\]
This completes the proof of the first claim regarding the convergence rate of \eqref{eq:viscHalpern}.

We now turn to its asymptotic regularity bound. Using the triangle inequality and the fact that \(x_k \in [f(x_{k-1}), T x_{k-1}]\), we obtain
\[
d(Tx_{k-1}, x_{k-1}) \le d(x_k, T x_{k-1}) + d(x_k, x_{k-1})
= \alpha_k\, d(Tx_{k-1}, f(x_{k-1})) + d(x_k, x_{k-1}).
\]
Recalling the definition of \(\alpha_{k}\), together with the bounds 
\(d(f(x_{k-1}), T x_{k-1}) \le 2C_{\bar x}\) and the first claim
\(d(x_k, x_{k-1}) \le \tfrac{2J\, C_{\bar x}}{(1-\beta)k}\), we deduce
\[
d(Tx_{k-1}, x_{k-1}) \le \frac{2}{k(1-\beta)}\,2C_{\bar x} + \frac{2J C_{\bar x}}{(1-\beta)k}
= \frac{2C_{\bar x}(J+2)}{(1-\beta)k}.
\]
Hence the second claim also holds, and both sequences \(d(x_k,x_{k-1})\) and 

\(d(Tx_{k-1},x_{k-1})\) decay at rate \(O(1/k)\).

\end{proof}

\medskip
This completes the analysis of the viscosity (Halpern-type) scheme in complete $CAT(0)$ spaces. 
%Together with the results obtained for the Krasnosel'ski\u{\i}--Mann iteration, we have thus confirmed $O(1/\sqrt{k})$ and $O(1/k)$ asymptotic regularity bounds for the two fundamental fixed-point algorithms in the nonlinear setting.
Together with the results obtained for the Krasnosel'ski\u{\i}--Mann iteration, this demonstrates that the proof strategies from the linear setting can be carried over to the nonlinear framework, yielding in particular the known \(O(1/\sqrt{k})\) and \(O(1/k)\) asymptotic regularity bounds for the two fundamental fixed-point algorithms.

\section{Hyperbolic HalpernGD-based Optimizer}
\label{sec:iperbolicHalpernGD}
In this section, we formulate a HalpernGD-type scheme on Hadamard spaces whose specialization to \(\mathbb H^n\) yields the Hyperbolic HalpernGD optimizer. The key idea is to replace the Euclidean forward step with the proximal (resolvent) mapping, which remains nonexpansive in the metric setting.
The properties of the proximal mapping recalled in Section~\ref{sec:preliminari} naturally motivate this construction.
In the Euclidean setting, the standard HalpernSGD iteration exploits the nonexpansivity of the forward operator 
$F = I - \eta\nabla f$, whose averaged nature is guaranteed by the Baillon–Haddad Theorem~\ref{Baillon-Haddad}.  
However, in nonlinear spaces it does not hold in general, and we provided a counterexample.
%\textcolor{red}{However, in nonlinear spaces—such as Hadamard spaces—the differential structure required by the forward operator is no longer available.}
\subsection{Failure of nonexpansivity of the forward operator on \(\mathbb H^2\)}
To the best of our knowledge, explicit examples of nonexpansive forward operators in non-Euclidean settings have not been investigated in the literature. The present section aims to fill this gap. 

In particular, we  will point out that, even on the hyperbolic plane, geodesic convexity and smoothness do not suffice to guarantee the nonexpansivity of the explicit forward step.\\
Let \(X\) be a general Hadamard manifold, let \(f\in C^1(X)\), and let \(\eta>0\); then the associated forward operator is defined by
\[
F_\eta(x):=\exp_x\bigl(-\eta \nabla f(x)\bigr), \qquad x\in X.
\]
Moreover, a function \(f\in C^2(X)\) is geodesically convex (see~Theorem $6.2$ in ~\cite{udriste94}) if and only if
\[
\operatorname{Hess}f_x(v,v)\ge 0
\qquad
\forall x\in X,\ \forall v\in T_xX.
\]
\begin{theorem}
\label{thm:forward-not-nonexpansive}
There exists a \(C^\infty\) geodesically convex function
\(f:\mathbb{H}^2\to\mathbb{R}\) with nonempty minimizer set such that, for every \(\eta>0\), the map \( F_\eta(x)=\exp_x(-\eta\nabla f(x)) \)
is not nonexpansive. Equivalently, for every \(\eta>0\) there exist \(x,y\in\mathbb H^2\) such that
\[
d_H\bigl(F_\eta(x),F_\eta(y)\bigr)>d_H(x,y).
\]
\end{theorem}
\begin{proof}
We work in the hyperboloid model of \(\mathbb H^2\) and use the global Fermi coordinates \((r,s)\in\mathbb R^2\) around the geodesic
$\Gamma=\{(\cosh s,\sinh s,0):s\in\mathbb R\}$,
given by
\[
\Phi(r,s):=(\cosh r\,\cosh s,\ \cosh r\,\sinh s,\ \sinh r).
\]
In these coordinates, the hyperbolic metric is
\(
\Phi^*g_H=dr^2+\cosh^2(r)\,ds^2.
\)
\\
Define \(f:\mathbb H^2\to\mathbb R\) by
\(
f(\Phi(r,s)):=\cosh r-1.
\)
Then \(f\ge 0\), and \(f=0\) if and only if \(r=0\). Hence
\(
\arg\min f=\Gamma.
\) Moreover,
\[
df=\sinh r\,dr,
\qquad
\nabla f=\sinh r\,\partial_r.
\]
and, by direct computation, we have:
\[
\operatorname{Hess}f(\partial_r,\partial_r)=\cosh r,
\quad
\operatorname{Hess}f(\partial_r,\partial_s)=0,
\quad
\operatorname{Hess}f(\partial_s,\partial_s)=\sinh^2 r\,\cosh r.
\]
Therefore \(\operatorname{Hess}f\) is positive semidefinite, and thus \(f\) is geodesically convex.
Fix \(\eta>0\). Since \(\displaystyle \frac{\sinh r}{r}\to+\infty\) as \(r\to+\infty\), we may choose \(r>0\) such that
\[
\eta \sinh r > 2r.
\]
Let \(a>0\), and set $x:=\Phi(r,0)$, $y:=\Phi(r,a).$
%\[ x:=\Phi(r,0),\qquad y:=\Phi(r,a). \]
For fixed \(s\in\mathbb R\), consider the curve
\[
\gamma(t):=\Phi(r-t\eta\sinh r,s),\qquad t\in\mathbb R.
\]
This is a geodesic, since it is an affine reparametrization of the curves 
\(u\mapsto \Phi(r-u,s)\) that are unit-speed geodesics orthogonal to $\Gamma$. Moreover,
\[
\gamma(0)=\Phi(r,s),
\qquad
\gamma'(0)=-\eta\sinh r\,\partial_r.
\]
From $\nabla f(\Phi(r,s))=\sinh r\,\partial_r$, it follows that
\[
F_\eta(\Phi(r,s))
=
%\exp_{\Phi(r,s)}(-\eta\nabla f(\Phi(r,s)))
%=
\exp_{\Phi(r,s)}(-\eta\sinh r\,\partial_r)
=
\gamma(1)
=
\Phi(r-\eta\sinh r,s).
\]
Therefore, $F_\eta(x)=\Phi(r-\eta\sinh r,0)$ and $F_\eta(y)=\Phi(r-\eta\sinh r,a)$.
\\
By the hyperboloid distance formula, it follows that
\[
\cosh d_H(x,y)
=
-\langle \Phi(r,0),\Phi(r,a)\rangle_L
=
1+\cosh^2 r\,(\cosh a-1),
\]
and analogously
\[
\cosh d_H(F_\eta(x),F_\eta(y))
=
1+\cosh^2(r-\eta\sinh r)\,(\cosh a-1).
\]
By the choice of \(r\), we have
\[
r-\eta\sinh r<-r,
\qquad\text{hence}\qquad
|r-\eta\sinh r|>r.
\]
Since \(\cosh\) is even and strictly increasing on \([0,+\infty)\), it follows that
\[
\cosh^2(r-\eta\sinh r)>\cosh^2 r.
\]
From \(\cosh a-1>0\), it follows that
\[
\cosh d_H(F_\eta(x),F_\eta(y))
>
\cosh d_H(x,y).
\]
Finally, since \(\operatorname{arcosh}\) is strictly increasing on \([1,+\infty)\),
\[
d_H(F_\eta(x),F_\eta(y))>d_H(x,y).
\]
Since \(\eta>0\) was arbitrary, this proves that for every \(\eta>0\) there exist
\(x,y\in\mathbb H^2\) such that
\[
d_H(F_\eta(x),F_\eta(y))>d_H(x,y).
\]
\end{proof}
\begin{remark}
The previous counterexample shows that geodesic convexity alone is not sufficient to guarantee nonexpansiveness of the explicit forward step on Hadamard manifolds. Moreover, since $f$ is $C^2$, its Hessian is bounded on every closed geodesic ball; hence $\nabla f$ is Lipschitz on every bounded convex region. Therefore, local smoothness on bounded convex subsets does not suffice either. Thus, in nonpositively curved geometry, the proximal mapping is not merely a convenient substitute for the forward operator, but in general the natural nonexpansive object. 
\end{remark}

In this context, the proximal (or resolvent) mapping introduced in Definition~\ref{def:proxHadamard} provides a geometric analogue of the forward step, preserving nonexpansivity in the metric sense.

\medskip
Following Bačák~\cite[Ch.~2, §2.2, Prop.~2.2.24, p.~44; Ch.~4, §4.2, Lemma~4.2.2, p.~73]{bacak2014convex}, 
the proximal mapping \( J_{\lambda f}(x) \) is well defined in every Hadamard space $X$:  by~\cite[Prop.~2.2.24]{bacak2014convex}, the minimizer exists and is unique for each $x\in X$; by~\cite[Lemma~4.2.2]{bacak2014convex}, the operator $J_{\lambda f}$ is both 1-Lipschitz and firmly nonexpansive.
Moreover, as shown in~\cite[Eq.~(2.2.24)]{bacak2014convex},
\[
x = J_{\lambda f}(x)
\quad\Longleftrightarrow\quad
x \in \arg\min f,
\]
so that the fixed points of $J_{\lambda f}$ coincide with the minimizers of $f$.  
This equivalence establishes $J_{\lambda f}$ as the natural nonlinear generalization of the Euclidean forward operator.

\medskip
Assume in addition that \(\arg\min f\neq\varnothing\).
Replacing the Euclidean forward step with the resolvent \(J_{\lambda f}\) naturally leads to a Hadamard-space analogue of the Halpern iteration; in particular, this yields the desired hyperbolic version when \(X=\mathbb H^n\). For any anchor point \(u\in X\) and any control sequence \(\{\alpha_k\}\subset(0,1)\), we obtain:
\begin{equation} \label{eq:hyperbolicHalpern}
    x_{k+1} = \alpha_{k+1} u \oplus (1-\alpha_{k+1}) J_{\lambda f}(x_k), \qquad x_0\in X.
\end{equation}
Since \(J_{\lambda f}\) is firmly nonexpansive, it is in particular nonexpansive. Hence, under the standard assumptions on \((\alpha_k)\), the sequence \((x_k)\) is asymptotically regular and converges strongly to the metric projection of \(u\) onto
\[
Fix(J_{\lambda f})=\arg\min f.
\]

\medskip
Iteration~\eqref{eq:hyperbolicHalpern} thus provides a theoretical foundation for a Hyperbolic Halpern--GD optimizer; indeed, convergence and quantitative results for related Halpern-type proximal point algorithms have already been established in Hilbert spaces and, more generally, in \(CAT(0)\) spaces \cite{leustean2021quantitative,sipos2022abstract}.  
It combines the geometric advantages of Hadamard spaces, such as the ability to represent hierarchical or tree-like data structures, with the accelerated convergence properties of Halpern-type iterations.  
From a practical standpoint, this construction bridges the gap between fixed-point theory in $CAT(0)$ spaces and modern geometric deep learning.  
In particular, the convergence guaranteed by the nonexpansivity of $J_{\lambda f}$ mirrors the behavior of HalpernSGD in Euclidean spaces, while adapting naturally to curved manifolds through the metric formulation of the proximal step.

\section{Conclusion and Future Work}
\label{sec:conclusione}
Our first main contribution lies in showing how proof strategies originally developed in Banach and Hilbert spaces, where linearity simplifies many arguments, can be successfully reformulated in the purely metric setting of $CAT(0)$ spaces.  
This approach not only recovers the classical asymptotic regularity bounds but also clarifies which steps of the convergence analysis depend essentially on the linear structure and which extend naturally to a general geodesic framework.
In addition, we introduced a hyperbolic optimization scheme, the Hyperbolic HalpernGD optimizer, obtained by replacing the Euclidean forward step with the proximal (resolvent) mapping in Hadamard spaces.
This construction provides a HalpernGD-type algorithm in nonpositively curved geometries, thereby yielding a principled optimizer for deep models operating in Hadamard spaces. As future work, we plan to implement and experimentally evaluate the stochastic counterpart of the Hyperbolic HalpernGD optimizer, comparing its performance with that of the standard Riemannian stochastic gradient descent (RSGD) widely used in hyperbolic deep learning.
From a theoretical standpoint, the asymptotic regularity rates suggest that Halpern-type methods should converge faster than Krasnosel’skii–Mann-type iterations, of which gradient descent is a particular instance.  
Therefore, we expect the Hyperbolic HalpernSGD optimizer to exhibit superior convergence speed and empirical efficiency compared with RSGD, in line with the improvements already observed in the Euclidean setting for HalpernSGD over classical SGD \cite{foglia2024halpernsgd,colao2025optimizer}.
Ultimately, this line of research aims to bridge fixed-point theory and convex analysis in metric spaces with hyperbolic deep learning, paving the way for faster and more geometrically consistent optimization algorithms on manifolds with curvature.
\\

\noindent\textbf{Acknowledgements}\\
\textit{This publication was partially funded by the PhD program in Mathematics and Computer Science at the University of Calabria, Cycle XXXVIII, with the support of a scholarship financed by DM 351/2022 (CUP H23C22000440007), within the NRRP funded by the European Union.}

% Please complete affiliations of all co-authors (using not abbreviated names and surnames).

\noindent Katherine Rossella Foglia (corresponding author)\\  %Name Surname
katherine.foglia@unical.it\\ \bigskip % ORCID identifier (optional)

\noindent {\small
\noindent University of Calabria\\
Department of Mathematics and Computer Science\\
Ponte P. Bucci, 30B, Arcavacata di Rende (CS), 87036, Italy
}\bigskip

\noindent Vittorio Colao \\
vittorio.colao@unical.it\\ \bigskip

\noindent {\small
\noindent University of Calabria\\
Department of Mathematics and Computer Science\\
Ponte P. Bucci, 30B, Arcavacata di Rende (CS), 87036, Italy
}\bigskip

\end{document}